\documentclass[a4paper,12pt]{article}

\usepackage[utf8]{inputenc}
\usepackage[english]{babel}
\usepackage{enumitem}
\usepackage{float}
\usepackage{amsmath}
\usepackage{amssymb}
\usepackage{amsthm}
\usepackage{amsfonts}
\usepackage{booktabs}
\usepackage{epsfig}
\usepackage{epstopdf}
\usepackage{cite}
\usepackage{multirow} 
\usepackage{algorithm,algpseudocode}
\usepackage{subcaption}
\usepackage{xcolor}
\usepackage[linkcolor=blue,colorlinks=true]{hyperref}
\let\originaleqref\eqref
\renewcommand{\eqref}{~\originaleqref}

\usepackage{caption}
\newtheorem{dfn}{Definition}[section]
\newtheorem{lem}[dfn]{Lemma}
\newtheorem{thm}[dfn]{Theorem}
\newtheorem{cor}[dfn]{Corollary}

\newtheorem{prop}[dfn]{Proposition}

\newtheorem{exm}[dfn]{Example}
\theoremstyle{definition}

\newtheorem{asm}[dfn]{Assumption}
\newtheorem{rem}[dfn]{Remark}

\newcommand{\R}{\mathbb{R}}

\allowdisplaybreaks

\title{A relaxed proximal point algorithm with double-inertial effects for nonconvex equilibrium problems}

\date{\today}

\author{
	Van Nam Tran\footnote{Faculty of Applied Sciences, Ho Chi Minh City University of Technology and Education, Ho Chi Minh City, Vietnam; e-mail: namtv@hcmute.edu.vn}	
	\and }
\begin{document}
	
	\maketitle

	\begin{abstract}
		In this paper, we present  a relaxation proximal point method with double inertial effects to approximate a solution of a non-convex equilibrium problem.  We give global
		convergence results of the iterative sequence generated by our algorithm. Some known results are recovered
		as special cases of our results. Numerical test is given to support the theoretical findings.

		\noindent  {\bf Keywords: Non-convexity, equilibrium problem, Proximal point algorithms, Two-step inertial extrapolation}
		
		\noindent {\bf 2010 MSC classification: 47H09, 65K15, 90C30} 
		
	\end{abstract}
	
	\section{Introduction}\label{Sec:Intro}
	\noindent
In this work  we are interested in the equilibrium problem (EP for short):
\begin{equation}
	\mbox{Find}  \quad \hat{x}\in C : f(\hat{x}; y)\geq 0; \quad \forall  y \in C; \tag{EP} 
\end{equation} where  $C$ is a nonempty set in $\R^n$ and $f:C\times C \longrightarrow  \R$ be a bifunction satisfying \(f(x, x)=0\) for all \(x\in C\). We denote by $S(C; f)$ the set of solutions of (EP).
	
Equilibrium problems were first studied by Ky Fan in \cite{FAN} and have been further developed over the last few decades (see, for instance, \cite{ANTI, ANS, BIGI2}). It has been shown that many problems, such as the optimization problem, the variational inequality, the saddle point problem, the Nash equilibrium problem in non-cooperative games, and the fixed point problem, can be formulated as equilibrium problems; see, for instance, \cite{ANTI, ANTI2, ANTI3, BLU, MUU} and references therein. In addition, numerous applications have been effectively described using the notion of equilibrium solutions, prompting extensive research efforts in this area. We refer the readers to the monograph \cite{BIGI2} for an excellent survey on the existence of equilibrium points and solution methods for finding them.

When studying algorithms for approximating solutions to equilibrium problems, the convexity condition plays a fundamental and indispensable role. In this paper, we introduce a relaxed proximal point method with dual inertial effects to address non-convex equilibrium problems. Our approach leverages the concept of strong quasiconvexity, a generalized form of convexity that includes a broader class of functions compared to strong convexity.
	
	Polyak introduced strongly quasiconvex functions in his seminal paper \cite{POL3}, which subsequently attracted attention from various researchers who explored them in different contexts, see, for instance, \cite{JOV, VIAL, GRA, IUS}. Recent studies delve into the application of strongly quasiconvex functions in optimization. For instance, \cite{CHE} examines the asymptotic convergence properties of saddle-point dynamics involving bifunctions that are strongly quasiconvex in one or both variables. Additionally, \cite{ROU} presents results on the asymptotic behavior of quasi-autonomous gradient expansive systems governed by a strongly quasiconvex function. This work reviews crucial existing results and initiates an initial exploration into minimizing members of this class using contemporary iterative methods. Specifically, it demonstrates the convergence of the classical proximal point algorithm under standard assumptions when employed for minimizing a strongly quasiconvex function. Similarly, in \cite{GRAD3}, the authors of this study reach a similar conclusion for a relaxed-inertial proximal point method.
	
	Various proximal point-type algorithms have been explored for solving equilibrium problems in the convex context. Some of these algorithms employ a bifunction endowed with a generalized monotonicity property rather than the classical monotonicity. However, there are only a few recent works that focus on iterative methods for solving equilibrium problems with the underlying function being (strongly) quasiconvex. Notably, in \cite{IUSEM}, one of the first proximal point methods within such a framework is proposed; \cite{MUU2} presents a subgradient method; \cite{Yen1} introduces a proximal subgradient algorithm, and the recent contribution in \cite{ANS} extends \cite{IUSEM} to incorporate Bregman distances.

	Polyak \cite{POL} initially introduced an inertial extrapolation, inspired by the heavy ball method applied to second-order dynamical systems with friction, as an acceleration technique for addressing smooth convex minimization problems. This method, akin to imparting inertia, aims to enhance the convergence rates of iterative algorithms. Numerous researchers have adopted inertial algorithms due to their significant impact on accelerating convergence rates in iterative algorithms, as evidenced by various studies \cite{CHEN, GRAD3, Cholamjiak, Padcharoen, Bot3, QUU} and relevant literature). In addition, algorithms with two inertial steps have demonstrated superior performance compared to those with a single inertial step. For instance,  the Douglas-Rachford splitting method with double-inertial of the form 
	\begin{equation*}
		x_{k+1}=F(x_k+\theta(x_k-x_{k-1})+\beta(x_{k-1}-x_{k-2}) )
	\end{equation*}
were proven to 	converge faster than Douglas-Rachford splitting method with the one-step inertial of the form 
	\begin{equation*}
		x_{k+1}=F(x_k+\theta(x_k-x_{k-1}) )
	\end{equation*} 
	in \cite[Section 4]{poon} .

	In \cite[Chapter 4]{Liang}, it was remarked that using a two-step inertial extrapolation with more than two points $x_k, x_{k-1}$ could give acceleration,
	for example,
	\begin{equation*}
		w_k=x_k+\theta(x_k-x_{k-1})+\beta(x_{k-1}-x_{k-2})
	\end{equation*} 
	with $\theta >0$ and $\beta  < 0.$

	Recently, there has been scholarly interest in multi-step inertial methods, as evidenced by studies conducted by authors, see, for example, \cite{Combettes, Nam}, and  \cite{nam3}. These investigations have yielded results demonstrating the efficacy of such approaches. The inadequacies of one-step inertial acceleration within the context of ADMM were discussed in \cite[section 3]{poon}, prompting the proposal of adaptive acceleration as an alternative. Furthermore, Polyak \cite{POL2} emphasized the potential of multi-step inertial methods to accelerate optimization processes. Consequently, it has been suggested that utilizing inertial methods involving more than two points may lead to improved performance in certain cases.
	

	\noindent
	\textbf{Contributions.}
	\begin{itemize}
		\item In this paper, we apply the two inertial step  proximal point method with relaxation to solve non-convex equilibrium problems. Our results extend the usage of the proximal point method from convex equilibrium problem  already studied in \cite{ANS,ANTI} to non-convex ones. 
		\item Instead one of the one-step inertial extrapolation already considered in several papers in \cite{Grad, GRAD3, Abubakar, Bot3,BotRadu,BotSedlmayerVuong,Cholamjiak,Padcharoen}  for convex equilibrium problems, we further accelerate our the proximal point method by adding a two-step inertial extrapolation.
		\item  A numerical experiment is given to show the benefits gained by considering a two-step inertial extrapolation instead of one-step inertial extrapolation.
	\end{itemize}
	
	\noindent
	\textbf{Organization}.
	\noindent We organize the rest of our paper as follows:  Section \ref{Sec:Prelims} contains basic definitions and results needed in subsequent sections. In Section \ref{Sec:Method}, we present and discuss our proposed method along with its convergence results. We give some numerical illustrations in Section \ref{Numerical} and concluding remarks are given in Section \ref{conclude}.

	\section{Preliminaries}\label{Sec:Prelims}
	\noindent In this section we recall some basic notions which are used frequently in the sequel. 
	
	We denote by $\langle \cdot,\cdot \rangle$ the \emph{inner product} and \(\|\cdot \| \) its induced norm of a Hilbert space \(\R^n\). Recall that   the effective domain of a function $g: \mathbb{R}^n \to \overline{\mathbb{R}}:=\mathbb{R} \cup \{\pm \infty\}$ is defined by $\mbox{dom}~g:=\{x\in \mathbb{R}: g(x) < +\infty\}.$  The function $g$ is said to be \emph{proper} if $g(x) > -\infty$ for every $x\in \mathbb{R}^n$ and $\mbox{dom}~g$ is nonempty (clearly, $g(x) = + \infty$ for every $x\notin \mbox{dom}~g$).  When the domain of $g$ is convex, we say that the function $g$ is 
	\begin{itemize}
		\item[(a)] \emph{convex} if, for any $x,y \in \mbox{dom}~g,$ then
		\begin{eqnarray*}
			g(\mu x + (1-\mu)y) \leq \mu g(x) + (1-\mu)g(y)  \ \ \ \forall \ \mu \in [0,1],
		\end{eqnarray*}
		\item[(b)] \emph{quasiconvex} if, for any $x,y \in \mbox{dom}~g,$ then
		\begin{eqnarray*}
			g(\mu x + (1-\mu)y) \leq \max\{g(x),g(y)\} \ \ \forall \ \mu \in [0,1],
		\end{eqnarray*}
		\item[(c)] \emph{strongly convex} (on dom~$g$) with modulus \(\gamma>0\) if there is a $\gamma\in (0, \infty) $ satisfying 
		$$	g(\mu x + (1-\mu)y) \leq  \mu g(y)+(1-\mu)g(x)-\mu(1-\mu)\frac{\gamma}{2}\|x-y\|^2,  \ \ \forall x, y\in \mbox{dom}~g, \ \forall  \mu \in [0,1],$$
		\item[(c)] \emph{strongly quasiconvex} (on dom~$g$) with modulus \(\gamma>0\) if there exists a \(\gamma\in (0, \infty) \) for which
		$$g(\mu y+(1-\mu)x)\leq \max\{g(y), g(x)\}-\mu(1-\mu)\frac{\gamma}{2}\|x-y\|^2, \forall x, y\in \mbox{dom}~g, \forall \mu\in [0,1]. $$
	\end{itemize}
	It is well known that a strongly convex function is also a strongly quasiconvex function; however, the inverse does not hold in general. For instance, the Euclidean norm is strongly quasiconvex on any bounded convex set but not strongly convex \cite[Theorem 2]{JOV}. Similarly, a convex function is also quasiconvex, while the inverse may not hold. For example, the function $g: \mathbb{R} \to \mathbb{R}$ with $g(x) = x^3,$ is quasiconvex but not convex. 	
	For more relationships between different notions of convexity, we refer readers to \cite{Grad}. For further studies on generalized convexity, see \cite{CAM,HAD}, among other works.
	
	The proximity operator becomes a very useful tool in studying numerical methods for xonvex optimization problems. We now recall this notion below.  
	
\noindent	The proximity operator of parameter $\lambda >0$ of a function $g: C \to \bar{\mathbb{R}}$ at $x \in \mathbb{R}^n$ is defined by ${\rm Prox}_{\gamma g}(C,\cdot ): \mathbb{R}^n \rightrightarrows \mathbb{R}^n$ with
	\begin{eqnarray*}
		{\rm Prox}_{\lambda g}(C, x) = \arg\min_{y\in C}\left\{g(y) + \frac{1}{2\lambda}\|y-x\|^2\right\}, \forall x\in \R^n,
	\end{eqnarray*}
	where $C$ is a nonempty subset of $\R^n$. When $C=\R^n$ we simply write $\mbox{\rm Prox}_{\lambda g}(\cdot)$.
	
It is worth noting that the sum of a strongly quasiconvex function and a half of the squared norm may not be  strongly quasiconvex, as shown in \cite[Remark 6]{LARA} and \cite{Grad}. As a result, the proximal operator of a strongly quasiconvex function is not a singleton, but a set-valued mapping. It is a singleton when the sum of a strongly quasiconvex function and a half of a squared norm is strongly quasiconvex.

	We now recall some notions related to bifunctions. 		Let $C$ be a closed and convex set in $\mathbb{R}^n,$  and $f: C\times C \to \mathbb{R}$ be a real-valued bifunction satisfying $f(x, x)=0, \ \ \forall x\in C$. We say that $f$ is
	\begin{itemize}
		\item[(a)] \emph{monotone} on $C,$ if for every $x,y\in C$ it holds that 
		\begin{eqnarray}
			f(x,y) + f(y,x) \leq 0,
		\end{eqnarray}
		\item[(b)] \emph{pseudononotone} on $C$ if $$f(x, y)\geq 0 \implies f(y, x)\leq 0, \forall x, y\in C.$$
	\end{itemize}

	\noindent 
	
	We list here some useful lemmas that are used for  analyzing the convergence of our proposed algorithm later. The first lemma presents a property of the proximal operator of a strongly quasiconvex function.

	\begin{lem}	(\cite[Proposition 6]{JOV}) Let $C$ be a closed and convex set in \(\R^n\),		$g:~\R^n \longrightarrow \overline{\R}$ be a proper, lower semicontinuous, strongly quasiconvex function with modulus $\gamma>0$ and such that $C\subseteq \mbox{\rm dom}\ g$, $\lambda>0$ and $x\in C$. If $\hat{x}\in \mbox{\rm Prox}_{\lambda g} (C, x)$, 
		then for all $y \in C$ and all $\mu\in [0; 1]$ one has
		\begin{equation*}
			g(\hat{x})-\max\{g(y), g(\hat{x})\}\leq \dfrac{\mu}{\lambda} \langle \hat{x}-x, y - \hat{x} \rangle +\dfrac{\mu}{2}\left(\dfrac{\mu}{\lambda}-\gamma+\mu \gamma\right) \|y-\hat{x}\|^2.
		\end{equation*}
	\end{lem}
	
	The following lemma presents several useful identities that will be frequently used in the subsequent sections.

	\begin{lem} \cite{Nam, Grad}\label{simple}
		Let $x,y,z \in \R^n$ and $a,b,\beta \in \mathbb{R}$. Then
		\begin{itemize}
			\item[(a)] We have that 	\begin{eqnarray*}
				&&\|(1+a)x-(a-b)y-bz\|^2\\
				&=& (1+a)\|x\|^2-(a-b)\|y\|^2-b\|z\|^2 +(1+a)(a-b)\|x-y\|^2\\
				&&+b(1+a)\|x-z\|^2-b(a-b)\|y-z\|^2.
			\end{eqnarray*}
			\item[(b)] The following identity holds\begin{eqnarray*}
				\langle x - z, y - x\rangle = \frac{1}{2}\|z - y\|^2 - \frac{1}{2}\|x - z\|^2 - \frac{1}{2}\|y-x\|^2.
			\end{eqnarray*}
			\item[(c)]  It holds	that \begin{eqnarray*}
				\|\beta x + (1-\beta)y\|^2 = \beta\|x\|^2 + (1-\beta)\|y\|^2 - \beta(1-\beta)\|x - y\|^2.
			\end{eqnarray*}
		\end{itemize}
	\end{lem}
	
	\section{The Algorithm and Convergence Analysis} \label{Sec:Method}
	\noindent
	In this section, we introduce and discuss our proposed method to solve non-convex equilibrium (EP). Furthermore, a convergence analysis of the proposed method is given.

	We first give some assumptions that are needed for investigating the convergence of the proposed algorithm for (EP).
	
	\begin{itemize}
		\item[(A1)] $y\longmapsto f(\cdot, y)$ is upper semicontinuous for all \(y\in C\);
		\item[(A2)] $f$ is pseudomonotone on \(C\);
		\item[(A3)] $f$ is lower semincontinuous (jointly in both arguments);
		\item[(A4)] $f(x,\cdot)$ is strongly quasiconvex on \(C\) with modulus \(\gamma>0\) for each \(x\in C\);
		\item[(A5)] $f$ satisfies the following Lipschitz-continuity type condition: there exists \(\eta>0\) such that
		\begin{equation*}
			f(x, z)-f(x, y)-f(y, z)\leq \eta \left(\|x-y\|^2+\|y-z\|^2\right) \mbox{ for all } x, y\in C.
		\end{equation*}
	\end{itemize}
	As mentioned in \cite{Grad} the assumptions (A1), (A2), (A3), (A5) are standard in the related literature, but \((A4)\) is weaker than the usual one which requires that $f$ is convex regarding the second argument. Also, authors in \cite{Grad} provided strongly quasiconvex functions which satisfy the assumptions (Ai), $i=1,2,3,4,5.$
	\\
	It is shown in \cite{Grad} that the following usual assumption in related literature: 
	
	\noindent	(A3')\ \  $f(x,\cdot)$ is lower semicontinuous for any \(x\in C\) 
	
	\noindent	is fulfilled when \((A3)\) is satisfied.

	From now on we always assume that the solution set \(S(C; f)\) of (EP) is nonempty. 
	\subsection{Proposed Method}

	\noindent We now present our proposed method.
	
	\begin{algorithm}[H]
		\caption{Relaxed Proximal Point Algorithm with double inertial for non-convex equilibrium Problems (RTIPPA-EP)}\label{alg1}
		\begin{algorithmic}[1]
			\State  Choose $\beta \in (-\infty,0]$, \(\rho\in[0,1)\),  $\theta \in [0,\frac{1}{2})$ and  $\lambda_k>0$ for all $k=1,2,3,\dots $. Pick $x_{-1},x_0,x_1\in \mathbb{R}^n$ and set $k=1.$
			\State  Given $x_{k-2}, x_{k-1}$ and $x_k$, compute
			\begin{align*}		
				y_k &= x_k + \theta (x_k - x_{k-1}) +\beta(x_{k-1} - x_{k-2})\\ 
				z_k &\in \mbox{argmin}_{x\in C}(f(y_k, x)+\frac{1}{2\lambda_k}\|y_k-x\|^2).
			\end{align*}
			\State 	If $y_k=z_k$ then STOP. Otherwise, go to Step 4
			\State 	 Choose some relaxation parameter \(\rho_k\in [1-\rho, 1+\rho]\) and update 		$$x_{k+1} =(1-\rho_k) y_k + \rho_k z_k$$
			\State 	 Let $k\leftarrow k+1$ and go to Step 2.
		\end{algorithmic}
	\end{algorithm}
	

	\begin{rem}
		When \(\beta=0\) our proposed algorithm reduces to Algorithm 1 of  \cite{Grad}. 
	\end{rem}

		To make it easier for convergence analysis of our proposed algorithm, we also make the following assumptions on parameter sequences \(\{\lambda_k\}\) and \(\{\rho_k\}\) and initial parameters $\theta\in [0, 1/2), \beta\in (-\infty, 0]$.
	\begin{asm}\label{CON} 
		
		Let $\alpha^k_{\max}=\dfrac{2-\rho_k}{\rho_k}$ 	 and 		
		\(\alpha^k_{\min}= \dfrac{2-4\eta\lambda_k-\rho_k}{\rho_k}.\)\\
		We assume that there exists $\epsilon > 0$ such that the following assumptions hold. 
		\begin{itemize}
			\item[(C1)]  $\dfrac{1}{\gamma-8\eta}<\lambda_k <\epsilon\leq \dfrac{1}{4\eta} \mbox{ for every}\quad  k\geq 0$; 
			\item[(C2)] $0<1-\rho\leq \rho_k\leq 1+\rho$ with $0\leq \rho\leq 1-4\eta\epsilon;$
			\item[(C3)] $	
			\begin{array}{ll}
				\max\left\{\dfrac{2\theta}{\alpha^k_{\min}} -(1-\theta); \dfrac{\theta}{1+\alpha^k_{\max}}-\alpha^k_{\min}\dfrac{(\theta-1)^2}{(1+\theta)(1+\alpha^k_{\max})} \right\}<\beta\leq 0;\\
			\end{array}$
			\item[(C4)] $	
			\begin{array}{ll}
				&\theta^2(1-\alpha^k_{\min})+\theta(1-2\beta+2\alpha^k_{\max}-2\alpha^k_{\min})-\beta(1-2\alpha^k_{\min})
				\\
				&+\beta^2(1-\alpha^k_{\min})-\alpha^k_{\min}<0.
			\end{array}$
		\end{itemize}
		
	\end{asm}

	\subsection{Convergence Analysis}
	\noindent In this Subsection, we present the convergence analysis of the sequence of iterates generated by our proposed Algorithm \ref{alg1}.
	We first show that the stopping criterion of Algorithm \ref{alg1} is valid. The proof is similar to the one of \cite[Proposition 3.2]{IUSEM}, hence is omitted.
	\begin{prop}
		Let $C$ be an affine subspace in $\R^n$, $\{\lambda_k\}_k$ and $\{\rho_k\}_k$ be sequences of positive numbers, $\{x_k\}_k, \{y_k\}_k$ and $\{z_k\}_k$ be the sequences generated by Algorithm \ref{alg1} and suppose that assumptions {\rm (A1), (A3')} and {\rm (A4)} are valid. If $y_k = z_k$ for some $k\in \mathbb{N}$, then $x_{k+1} = y_k$ is a solution of (EP), i.e., $\{y_k\} = S(C; f)$.
	\end{prop}
	
	The next result is the first step for analysis of convergence. Its proof is similar to the one in \cite[Proposition 6]{Grad}

	\begin{prop}\label{pro1}
		Let $C$ be an affine subspace in $\R^n$, $\{\lambda_k\}_k$ and $\{\rho_k\}_k$ be sequences of positive numbers, $\{x_k\}_k, \{y_k\}_k$ and $\{z_k\}_k$ be the sequences generated by Algorithm \ref{alg1} and suppose that assumptions {\rm (Ai)} with $i = 1, 2, 3',  4, 5$ hold.
		Let $\hat{x} \in S(C; f)$. Then for every $k\geq 0,$ at least one of the following inequalities holds:
		\begin{equation}\label{ineq 3.5}
			\|x_{k+1}-\hat{x}\leq \|y_k-\hat{x}\|^2 -\dfrac{2-\rho_k}{\rho_k}\|x_{k+1}-y_k\|^2-\dfrac{\rho_k(\gamma\lambda_k-1)}{2}\|z_k-\hat{x}\|^2,
		\end{equation}
		\begin{equation}\label{ineq 3.6}
			\|x_{k+1}-\hat{x}\|^2\leq \|y_k-\hat{x}\|^2-\dfrac{2-4\eta \lambda_k-\rho_k}{\rho_k}\|x_{k+1}-y_k\|^2-\dfrac{\rho_k(\gamma\lambda_k-1-8\eta\lambda_k)}{2}\|z_k-\hat{x}\|^2.
		\end{equation}
	\end{prop}
	As shown in Remark 7 of \cite{Grad}, under Assumptions \eqref{CON}, the last terms on the right-hand side of inequalities \eqref{ineq 3.5} and \eqref{ineq 3.6} are nonnegative, i.e., 
	$$\dfrac{\rho_k(\gamma\lambda_k-1)}{2}\|z_k-\hat{x}\|^2\geq 0 \mbox{ and } \dfrac{\rho_k(\gamma\lambda_k-1-8\eta\lambda_k)}{2}\|z_k-\hat{x}\|^2\geq 0.$$

	Because of these facts,  we obtain the followings. 
	\begin{cor}
		Let $C$ be an affine subspace in $R^n$, $f$ be such that assumptions (Ai)
		with $i = 1, 2, 3', 4, 5$ hold, $\{\lambda_k\}_k$ and $\{\rho_k\}_k$ be sequences of positive numbers 	such that assumption (C1)  holds, $\{x_k\}_k, \{y_k\}_k$ and $\{z_k\}_k$ be the sequences generated by Algorithm \ref{alg1}. Taking $ \{\hat{x}\}=S(C;f))$, then for every $k \geq 0$, at least one of the
		following inequalities holds
		\begin{equation}\label{dg1} 
			\|x_{k+1}-\hat{x}\leq \|y_k-\hat{x}\|^2 -\dfrac{2-\rho_k}{\rho_k}\|x_{k+1}-y_k\|^2,
		\end{equation}
		\begin{equation}\label{dg2}
			\|x_{k+1}-\hat{x}\|^2\leq \|y_k-\hat{x}\|^2-\dfrac{2-4\eta \lambda_k-\rho_k}{\rho_k}\|x_{k+1}-y_k\|^2.
		\end{equation}
		
	\end{cor}
	
	The lemma below shows the boundedness of the sequence of iterates $\{x_k\}$  generated by Algorithm \ref{alg1}.
	
	\begin{lem}\label{LEM} Let C be an affine subspace in $\R^n$, f be such that assumptions (Ai)
		$i = 1; 2; 3; 4; 5$ hold. $\{\lambda_k\}_k$ and $\{\rho_k\}_k$ be sequences of positive numbers such that assumptions (Ci) $(i=1, 2)$ hold, the inertial parameters $\theta\in [0, 1/2)$ and $\beta\leq 0$ satisfy $Ci$, $i=(3, 4)$ and  $\{x_k\}_k, \{y_k\}_k$ and $\{z_k\}_k$ be the sequences generated by Algorithm \ref{alg1}. 
		Then 
		
		\begin{enumerate}
			\item [(a)] We have that \begin{eqnarray}\label{ade10}
				\underset{k\rightarrow \infty}\lim \|x_{k-1}-x_{k-2}\|=\lim\limits_{k\to \infty}\|x_{k+1} - y_k\|=\lim\limits_{k\to \infty} \|x_k - y_k\|=0.
			\end{eqnarray}
			\item [(b)] It holds that  \begin{eqnarray}\label{SHE4}
				\lim_{k\to \infty}\|y_k - z_k\| = 0
			\end{eqnarray}
			\item [(c)]The sequences $\{x_k\}_k, \{y_k\}_k, \{z_k\}_k$ are all bounded.
		\end{enumerate} 
	\end{lem}
	
	\begin{proof}
		
		We obtain from \eqref{dg1} and \eqref{dg2}  that one of the following inequality holds. 
		\begin{eqnarray}\label{SHEU7}
			\|x_{k+1} - \hat{x}\|^2 \leq \|y_k - \hat{x}\|^2 - \dfrac{2-\rho_k}{\rho_k} \|x_{k+1} - y_k\|^2
		\end{eqnarray}
	or 
			\begin{eqnarray}\label{SHEU7n}
			\|x_{k+1} - \hat{x}\|^2 \leq \|y_k - \hat{x}\|^2 - \dfrac{2-4\eta \lambda_k-\rho_k}{\rho_k} \|x_{k+1} - y_k\|^2
		\end{eqnarray}		
Assume that \eqref{SHEU7} holds. Let $\alpha_k=\dfrac{2-\rho_k}{\rho_k}$.		Observe that under assumptions (C1) and (C2), \(\alpha_k>0\) and 
		\begin{eqnarray*}
			y_k - \hat{x} &=& x_k + \theta(x_k - x_{k-1}) + \beta(x_{k-1} - x_{k-2}) - \hat{x}\\
			&=& (1+\theta)(x_k - \hat{x}) - (\theta - \beta)(x_{k-1} - \hat{x}) - \beta(x_{k-2} - \hat{x}).
		\end{eqnarray*}
		Hence by Lemma \ref{simple}, we have
		\begin{eqnarray}\label{SHE}
			\|y_k - \hat{x}\|^2 &=& \|(1+\theta)(x_k - \hat{x}) - (\theta - \beta)(x_{k-1} - \hat{x}) - \beta(x_{k-2} - \hat{x})\|^2\nonumber\nonumber\\
			&=& (1+\theta)\|x_k - \hat{x}\|^2 - (\theta - \beta)\|x_{k-1} - \hat{x}\|^2 - \beta\|x_{k-2} - \hat{x}\|^2\nonumber\\
			&&\;\;+ (1+\theta)(\theta - \beta)\|x_k - x_{k-1}\|^2 + \beta(1+\theta)\|x_k - x_{k-2}\|^2\nonumber\\
			&&\;\; -\beta(\theta - \beta)\|x_{k-1} - x_{k-2}\|^2.
		\end{eqnarray}
		Observe that
		\begin{eqnarray*}
			2\theta \langle x_{k+1}-x_k, x_k-x_{k-1}\rangle &=& 2 \langle \theta(x_{k+1}-x_k), x_k-x_{k-1}\rangle \nonumber \\
			&\leq&2|\theta| \|x_{k+1}-x_k\|\|x_k-x_{k-1}\|\nonumber\\
			&=&2\theta \|x_{k+1}-x_k\|\|x_k-x_{k-1}\|
		\end{eqnarray*}
		and so
		\begin{equation}\label{happy1}
			-2\theta \langle x_{k+1}-x_k, x_k-x_{k-1}\rangle \geq -2\theta \|x_{k+1}-x_k\|\|x_k-x_{k-1}\|.
		\end{equation}
		Also,
		\begin{eqnarray*}
			2\beta \langle x_{k+1}-x_k,x_{k-1}-x_{k-2} \rangle &=& 2\langle \beta(x_{k+1}-x_k),x_{k-1}-x_{k-2} \rangle \nonumber \\
			&\leq&2|\beta| \|x_{k+1}-x_k\|\|x_{k-1}-x_{k-2}\|
		\end{eqnarray*}
		which implies that
		\begin{equation}\label{happy2}
			-2\beta \langle x_{k+1}-x_k,x_{k-1}-x_{k-2} \rangle  \geq -2|\beta| \|x_{k+1}-x_k\|\|x_{k-1}-x_{k-2}\|.
		\end{equation}
		Similarly, we note that
		\begin{eqnarray*}
			2\beta\theta \langle x_{k-1}-x_k,x_{k-1}-x_{k-2}\rangle &=& 2\langle \beta\theta(x_{k-1}-x_k),x_{k-1}-x_{k-2}\rangle \nonumber \\
			&\leq&2|\beta|\theta \|x_{k-1}-x_k\|\|x_{k-1}-x_{k-2}\|\nonumber\\
			&=&2|\beta|\theta \|x_k-x_{k-1}\|\|x_{k-1}-x_{k-2}\|
		\end{eqnarray*}
		and thus,
		\begin{eqnarray}\label{happy3}
			2\beta\theta \langle x_k-x_{k-1},x_{k-1}-x_{k-2}\rangle &=& -2 \beta\theta \langle x_{k-1}-x_k,x_{k-1}-x_{k-2}\rangle  \nonumber \\
			&\geq& -2|\beta|\theta \|x_k-x_{k-1}\|\|x_{k-1}-x_{k-2}\|.
		\end{eqnarray}		
		By \eqref{happy1}, \eqref{happy2} and \eqref{happy3},  and Cauchy-Schwarz inequality one has 
		\begin{eqnarray}\label{SHEU9}
			\|x_{k+1}-z_k\|^2 &=& \|x_{k+1}-(x_k+\theta(x_k-x_{k-1})+\beta(x_{k-1}-x_{k-2}))\|^2 \nonumber\\
			&=&\|x_{k+1}-x_k-\theta(x_k-x_{k-1})-\beta(x_{k-1}-x_{k-2})\|^2 \nonumber\\
			&=&\|x_{k+1}-x_k\|^2-2\theta \langle x_{k+1}-x_k, x_k-x_{k-1}\rangle \nonumber\\
			&&-2\beta \langle x_{k+1}-x_k,x_{k-1}-x_{k-2} \rangle+\theta^2\|x_k-x_{k-1}\|^2\nonumber\\
			&&+2\beta\theta \langle x_k-x_{k-1},x_{k-1}-x_{k-2}\rangle+\beta^2\|x_{k-1}-x_{k-2}\|^2\nonumber\\
			&\geq& \|x_{k+1}-x_k\|^2-2\theta \|x_{k+1}-x_k\|\|x_k-x_{k-1}\|\nonumber\\
			&&-2|\beta| \|x_{k+1}-x_k\|\|x_{k-1}-x_{k-2}\|+\theta^2\|x_k-x_{k-1}\|^2\nonumber\\
			&&-2|\beta|\theta \|x_k-x_{k-1}\|\|x_{k-1}-x_{k-2}\|+\beta^2\|x_{k-1}-x_{k-2}\|^2\nonumber\\
			&\geq& \|x_{k+1}-x_k\|^2-\theta \|x_{k+1}-x_k\|^2-\theta\|x_k-x_{k-1}\|^2\nonumber\\
			&&-|\beta| \|x_{k+1}-x_k\|^2-|\beta|\|x_{k-1}-x_{k-2}\|^2+\theta^2\|x_k-x_{k-1}\|^2\nonumber\\
			&&-|\beta|\theta \|x_k-x_{k-1}\|^2-|\beta|\theta\|x_{k-1}-x_{k-2}\|^2+\beta^2\|x_{k-1}-x_{k-2}\|^2\nonumber\\
			&=&(1-|\beta|-\theta)\|x_{k+1}-x_k\|^2+(\theta^2-\theta-|\beta|\theta)\|x_k-x_{k-1}\|^2\nonumber\\
			&&+(\beta^2-|\beta|-|\beta|\theta)\|x_{k-1}-x_{k-2}\|^2.
		\end{eqnarray}
		Combining \eqref{SHE} and \eqref{SHEU9} in \eqref{SHEU7}, we obtain (noting that $\beta \leq 0$)
		\begin{eqnarray*}
			\|x_{k+1}-\hat{x}\|^2 &\leq&(1+\theta)\|x_k-\hat{x}\|^2-(\theta-\beta)\|x_{k-1}-\hat{x}\|^2-\beta\|x_{k-2}-\hat{x}\|^2\nonumber \\
			&&+(1+\theta)(\theta-\beta)\|x_k-x_{k-1}\|^2+\beta(1+\theta)\|x_k-x_{k-2}\|^2\nonumber \\
			&&-\beta(\theta-\beta)\|x_{k-1}-x_{k-2}\|^2
			-\alpha_k(1-|\beta|-\theta)\|x_{k+1}-x_k\|^2\nonumber \\
			&&-\alpha_k(\theta^2-\theta-|\beta|\theta)\|x_k-x_{k-1}\|^2 \nonumber\\
			&&-\alpha_k(\beta^2-|\beta|-|\beta|\theta)\|x_{k-1}-x_{k-2}\|^2\nonumber \\
			&=& (1+\theta)\|x_k-\hat{x}\|^2-(\theta-\beta)\|x_{k-1}-\hat{x}\|^2-\beta\|x_{k-2}-\hat{x}\|^2\nonumber \\
			&&+\Big((1+\theta)(\theta-\beta)-\alpha_k (\theta^2-\theta-|\beta|\theta)\Big)\|x_k-x_{k-1}\|^2\nonumber \\
			&&+\beta(1+\theta)\|x_k-x_{k-2}\|^2- \alpha_k(1-|\beta|-\theta)\|x_{k+1}-x_k\|^2\nonumber \\
			&&-\Big(\beta(\theta-\beta)+ \alpha_k (\beta^2-|\beta|-|\beta|\theta)\Big)\|x_{k-1}-x_{k-2}\|^2\nonumber \\
			&\leq& (1+\theta)\|x_k-\hat{x}\|^2-(\theta-\beta)\|x_{k-1}-\hat{x}\|^2-\beta\|x_{k-2}-\hat{x}\|^2\nonumber \\
			&&+\Big((1+\theta)(\theta-\beta)-\alpha_k (\theta^2-\theta+\beta\theta)\Big)\|x_k-x_{k-1}\|^2\nonumber \\
			&&-\alpha_k (1+\beta-\theta)\|x_{k+1}-x_k\|^2\nonumber \\
			&&-\Big(\beta(\theta-\beta)+\alpha_k (\beta^2+\beta+\beta\theta)\Big)\|x_{k-1}-x_{k-2}\|^2.
		\end{eqnarray*}
		Rearranging, we get
		\begin{eqnarray}\label{SHEU14}
			&& \|x_{k+1} - \hat{x}\|^2 - \theta\|x_k - \hat{x}\|^2 - \beta\|x_{k-1} - \hat{x}\|^2\nonumber\\
			&&\;+\Big(\frac{\alpha-L}{\alpha+L} \Big)(1+\beta-\theta)\|x_{k+1}-x_k\|^2\nonumber\\
			&&\leq \|x_k - \hat{x}\|^2 - \theta\|x_{k-1} - \hat{x}\|^2 - \beta\|x_{k-2} - \hat{x}\|^2\nonumber\\
			&&+\Big((1+\theta)(\theta-\beta)-\alpha_k (\theta^2-2\theta+\beta\theta+\beta+1)\Big)\|x_k-x_{k-1}\|^2\nonumber \\
			&&-\Big(\beta(\theta-\beta)+\alpha_k(\beta^2+\beta+\beta\theta)\Big)\|x_{k-1}-x_{k-2}\|^2 \nonumber \\
			&&+\alpha_k(1+\beta-\theta)\|x_k-x_{k-1}\|^2.
		\end{eqnarray}
		Define
		\begin{eqnarray*}
			\Gamma_k &:=& \|x_k - \hat{x}\|^2 - \theta\|x_{k-1} - \hat{x}\|^2 - \beta\|x_{k-2} - \hat{x}\|^2\\
			&&\;\;+ \alpha_k(1+\beta-\theta)\|x_k - x_{k-1}\|^2.
		\end{eqnarray*}
		Let us show that $\Gamma_k \geq 0, \ \ \forall n\geq 1.$ Now,
		\begin{eqnarray}\label{SHE1}
			\Gamma_k &=&\|x_k - \hat{x}\|^2 - \theta\|x_{k-1} - \hat{x}\|^2 - \beta\|x_{k-2} - \hat{x}\|^2\nonumber\\
			&&\;\;+  \alpha_k(1+\beta-\theta)\|x_k - x_{k-1}\|^2\nonumber\\
			&\geq& \|x_k - \hat{x}\|^2 - 2\theta\|x_k - x_{k-1}\|^2 - 2\theta\|x_k - \hat{x}\|^2\nonumber\\
			&&\;- \beta\|x_{k-2} - \hat{x}\|^2 +  \alpha_k(1+\beta-\theta)\|x_k - x_{k-1}\|^2\nonumber\\
			&=& (1-2\theta)\|x_k - \hat{x}\|^2 + \Big[ \alpha_k(1+\beta-\theta) - 2\theta\Big]\|x_k - x_{k-1}\|^2\nonumber\\
			&&\;\;- \beta\|x_{k-2} - \hat{x}\|^2.
		\end{eqnarray}
		Since $\theta<\frac{1}{2}, \beta \leq 0$ and     by Assumption \ref{CON}, (C1) and (C3) , it follows from \eqref{SHE1}  that  $\Gamma_k \geq 0, ~~\forall k\geq 1.$ Furthermore, we derive from \eqref{SHEU14} that
		\begin{eqnarray}\label{SHEU15}
			\Gamma_{k+1}-\Gamma_k &\leq& \left((1+\theta)(\theta-\beta)-\alpha_k\left(\theta^2-2\theta+\beta\theta+\beta+1\right)\right)\|x_k-x_{k-1}\|^2\nonumber \\
			&&-\Big(\beta(\theta-\beta)+\alpha_k(\beta^2+\beta+\beta\theta)\Big)\|x_{k-1}-x_{k-2}\|^2 \nonumber \\
			&&=-\left((1+\theta)(\theta-\beta)-\alpha_k(\theta^2-2\theta+\beta\theta+\beta+1)\right)\Big(\|x_{k-1}-x_{k-2}\|^2 \nonumber \\
			&&-\|x_k-x_{k-1}\|^2\Big)+ \Big((1+\theta)(\theta-\beta)-\alpha_k \left(\theta^2-2\theta+\beta\theta+\beta+1\right)\nonumber\\
			&&-\beta\left(\theta-\beta\right)-\alpha_k \left(\beta^2+\beta+\beta\theta\right) \Big) \|x_{k-1}-x_{k-2}\|^2\nonumber \\
			&=&c_1\left(\|x_{k-1}-x_{k-2}\|^2-\|x_k-x_{k-1}\|^2\right)-c_2\|x_{k-1}-x_{k-2}\|^2,
		\end{eqnarray}
		where
		\begin{eqnarray*}
			c_1:=-\left((\theta -\beta)(1+\theta) - \alpha_k(\theta^2 - 2\theta +\beta\theta +\beta + 1)\right)
		\end{eqnarray*}
		\noindent
		and
		\begin{eqnarray*}
			&&c_2:=- \Big((\theta -\beta)(1+\theta) - \alpha_k(\theta^2 - 2\theta +\beta\theta +\beta + 1)\nonumber\\
			&&-\beta(\theta -\beta)-\alpha_k(\beta^2 +\beta+\beta\theta)\Big).
		\end{eqnarray*}
		\noindent
		\noindent  By Assumption (C3) in \ref{CON} (i), it holds that 
		$c_1>0$. Also $c_2>0$ by Assumption \ref{CON} (ii).
		By \eqref{SHEU15}, we have 
		\begin{eqnarray}\label{ade7}
			\Gamma_{k+1}+c_1\|x_k-x_{k-1}\|^2&\leq& \Gamma_{k}+c_1\|x_{k-1}-x_{k-2}\|^2\nonumber \\
			&&-c_2\|x_{k-1}-x_{k-2}\|^2.
		\end{eqnarray}
		Letting $\bar{\Gamma}_k:=\Gamma_k+c_1\|x_{k-1}-x_{k-2}\|^2$. Then \(\bar{\Gamma}_k\geq 0\) for all \(k\geq 1\) since $\Gamma_k\geq 0$ for all $k\geq 1$. Also,  it follows  from \eqref{ade7} that
		\begin{equation*}
			\bar{\Gamma}_{k+1} \leq  \bar{\Gamma}_{k} \ \ \mbox{ for all $k\geq $.}
		\end{equation*}
		Thus, the sequence $\{\overline{\Gamma}_{k}\}$ is non-increasing and  bounded from below,  and thus $\underset{n\rightarrow \infty}\lim \overline{\Gamma}_{k}$ exists.  Consequently, we get from \eqref{ade7} and the squeeze theorem that
		\begin{eqnarray*}
			\underset{n\rightarrow \infty}\lim  c_2\|x_{k-1}-x_{k-2}\|^2=0.
		\end{eqnarray*}
		Hence,
		(a) \begin{eqnarray}\label{nam1} 
			\underset{n\rightarrow \infty}\lim \|x_{k-1}-x_{k-2}\|=0.
		\end{eqnarray}
		As a result,
		\begin{eqnarray}\label{SHE2}
			\|x_{k+1} - y_k\| &=& \|x_{k+1} -x_k - \theta(x_k - x_{k-1}) - \beta(x_{k-1} - x_{k-2})\|\nonumber\\
			&\leq& \|x_{k+1} - x_k\| + \theta\|x_k - x_{k-1}\| + |\beta|\|x_{k-1} - x_{k-2}\| \to 0
		\end{eqnarray}
		as $k \to \infty.$
		By $\underset{n\to \infty}\lim \|x_{k+1} - x_k\| = 0$ one has 
		\begin{eqnarray*}
			\|x_k - y_k\| \leq \|x_k - x_{k+1}\| + \|x_{k+1} - y_k\| \to 0, \ \ k \to \infty.
		\end{eqnarray*}
		
		(b) We have $x_{k+1}-y_k=\rho_k(y_k-z_k)$. It follows from \eqref{SHE2} that $\lim\limits_{k\to\infty} \|y_k-z_k\|=0$. 
		
		(c)  By definition of $\overline{\Gamma}$, \eqref{nam1} and the existence of $\lim\limits_{k\to\infty} \overline{\Gamma}_k$, we have that $\lim\limits_{k\to \infty} \Gamma_k$ exists and hence $\{\Gamma_k\}_k$ is bounded.
		Now, since $\lim\limits_{k\to\infty} \|x_k-x_{k-1}\|=0$, we have from the definition of $\Gamma_k$ that
		\begin{equation}\label{beh}
			\lim_{k\to \infty}\left[\|x_k - \hat{x}\|^2 - \theta\|x_{k-1} - \hat{x}\|^2 - \beta\|x_{k-2} - \hat{x}\|^2\right] \ \mbox{exists}.
		\end{equation}
		Using the boundedness of $\{\Gamma_{k}\}_k$, we obtain from \eqref{SHE1} that $\{x_k\}_k$ is bounded. Consequently, both $\{z_k\}_k$ and $\{y_k\}_k$ are also bounded.
		
If the equality \eqref{SHEU7n} holds, by repeating the  argument as above we get the same conclusion. 		
		
		Therefore the proof is completed. 
	\end{proof}

	\noindent
	Our convergence result for Algorithm \ref{alg1} is given next.
	
	\begin{thm}\label{Sheu11}
		Let C be an affine subspace in $\R^n$, f be such that assumptions (Ai)
		$i = 1; 2; 3; 4; 5$ hold. Let $\{\lambda_k\}_k$ and $\{\rho_k\}_k$ be sequences of positive numbers such that assumptions (Ci) $(i=1, 2)$ hold and the inertial parameters $\theta\in [0, 1/2)$ and $\beta\leq 0$ satisfy $(Ci)$, $(i=3, 4)$. 
		Then the sequence  $\{x_k\}_k$ generated by Algorithm \ref{alg1} converges to a solution of (EP).
	\end{thm}
	
	\begin{proof}
		By Lemma \ref{LEM}, we have that $\{x_k\}_k$ is bounded. Suppose $x^*$ is an accumulating point of $\{x_k\}_k$.  	From	\eqref{ade10} and \eqref{SHE4} we derive that $x^*$ is also an accumulating point of $\{y_k\}_k$ and of $\{z_k\}_k$, .i.e, 
		\begin{equation}\label{eq3.17}
			\lim\limits_{m\to\infty} y_{k_m}=\lim\limits_{m\to\infty} z_{k_m}=x^*.
		\end{equation}
		
		\noindent From the definition of $z_k$  we have that for any \(k\geq 0\) 
		\begin{equation*}
			f(y_k,z_k)+\dfrac{1}{2\lambda_k}\|y_k-z_k\|^2\leq f(y_k, x)+\dfrac{1}{2\lambda_k}\|y_k-x\|^2, \forall x\in C.
		\end{equation*}
		It follows that 
		\begin{equation*}
			f(y_k,z_k)\leq f(y_k, x)+\dfrac{1}{2\lambda_k}\|y_k-x\|, \forall x\in C.
		\end{equation*}
		Take $x=\mu y+(1-\mu)x^*$ with \(y\in C\) and \(\mu\in [0,1]\). Then 
		$$\begin{array}{llll}
			f(y_k, z_k) &\leq & f(y_k, \mu y+(1-\mu)x^*)+\dfrac{1}{2\lambda_k}\|\mu(y_k-y)+(1-\mu)(y_k-x^*)\|^2\\
			& \leq & \max\{f(y_k, y), f(y_k, x^*)\} -\dfrac{\mu(1-\mu)\gamma}{2}\|y-x^*\|^2 
			\\
			&+& \dfrac{\mu}{2\lambda_k}\|y_k-y\|^2+\dfrac{1-\mu}{2\lambda_k}\|y_k-x^*\|^2-\dfrac{\mu(1-\mu)}{2\lambda_k}\|y-x^*\|^2\\
			&=& \max\{f(y_k,y),f(y_k, x^*)\}+\dfrac{\mu}{\lambda_k}\langle x^*-y_k, y-x^*\rangle +\dfrac{1}{2\lambda_k}\|y_k-x^*\|^2\\
			&+& \dfrac{\mu}{2}\left( \dfrac{\mu}{\lambda_k}-\mu +\mu\gamma\right)\|y-x^*\|^2, \forall y\in C, \forall \mu \in [0, 1].
		\end{array}$$
		Take \(\lambda^*:=1/(\gamma-8\eta)\). Substituting \(k\) by \(k_m\) in the above inequality and then take $\limsup\limits_{t\to \infty} $ and using \eqref{eq3.17} and assumption (A3)  one has 
		\begin{equation*}
			\begin{array}{lll}
				0&=&f(x^*, x^*)\leq \liminf\limits_{m\to \infty} f(y_{k_m}, z_{k_m})\leq \limsup\limits_{m\to \infty} f(y_{k_m}, z_{k_m})\\
				&\leq & \max\{f(x^*, y), 0\}+\dfrac{\mu}{2}\left( \dfrac{\mu}{\lambda^*}-\gamma+\gamma\mu\right) \|y-x^*\|^2, \forall y\in C, \forall \mu\in [0,1].
			\end{array}
		\end{equation*}
		Since $\gamma>0$ we can take $\mu<(\lambda^*\gamma)/(1+\lambda^*\gamma)$ such that $\mu/\lambda^*-\gamma+\mu\gamma<0$. Then 
		$$0\leq \max\{f(x^*, y), 0\} =f(x^*, y),\quad \forall y\in C\setminus\{x^*\}.$$
		Thus, $x^*\in S(C, f).$\\
		
		\noindent  We suppose now that there exist $\{x_{kj}\} \subset \{x_k\}$ and $\{x_{k_m}\} \subset \{x_k\}$ such that $x_{k_j}\rightarrow \bar{x}, j\rightarrow \infty$ and $x_{k_m}\rightarrow x^*, m\rightarrow \infty$. We  will show that $\bar{x}=x^*$.\\
		
		\noindent Observe that
		\begin{equation}\label{away1}
			2\langle x_k,x^*-\bar{x}\rangle =\|x_k-\bar{x}\|^2-\|x_k-x^*\|^2-\|\bar{x}\|^2+\|x^*\|^2,
		\end{equation}
		
		\begin{equation*}\label{away2}
			2\langle x_{k-1},x^*-\bar{x}\rangle =\|x_{k-1}-\bar{x}\|^2-\|x_{k-1}-x^*\|^2-\|\bar{x}\|^2+\|x^*\|^2,
		\end{equation*}
		and
		\begin{equation*}\label{away2a}
			2\langle x_{k-2},x^*-\bar{x}\rangle =\|x_{k-2}-\bar{x}\|^2-\|x_{k-2}-x^*\|^2-\|\bar{x}\|^2+\|x^*\|^2,
		\end{equation*}
		Therefore,
		\begin{eqnarray}\label{away3}
			2\langle -\theta x_{k-1},x^*-\bar{x}\rangle &=&-\theta\|x_{k-1}-\bar{x}\|^2+\theta\|x_{k-1}-x^*\|^2\nonumber \\
			&&+\theta\|\bar{x}\|^2-\theta\|x^*\|^2
		\end{eqnarray}
		and
		\begin{eqnarray}\label{away3a}
			2\langle -\beta x_{k-2},x^*-\bar{x}\rangle &=&-\beta\|x_{k-2}-\bar{x}\|^2+\beta\|x_{k-2}-x^*\|^2\nonumber \\
			&&+\beta\|\bar{x}\|^2-\beta\|x^*\|^2
		\end{eqnarray}
		Addition of \eqref{away1}, \eqref{away3} and \eqref{away3a} gives
		\begin{eqnarray*}
			& & 2\langle x_k-\theta x_{k-1}-\beta x_{k-2},x^*-\bar{x}\rangle= \Big(\|x_k-\bar{x}\|^2-\theta \|x_{k-1}-\bar{x}\|^2-\beta \|x_{k-2}-\bar{x}\|^2 \Big) \\
			&&-\Big(\|x_k-x^*\|^2-\theta \|x_{k-1}-x^*\|^2-\beta \|x_{k-2}-x^*\|^2 \Big)
			+(1-\theta-\beta)(\|x^*\|^2-\|\bar{x}\|^2).
		\end{eqnarray*}
		According to \eqref{beh}, we get that 
		\begin{eqnarray*}
			\underset{k\rightarrow \infty}\lim \Big[\|x_k-x^*\|^2 -\theta\|x_{k-1}-x^*\|^2-\beta \|x_{k-2}-x^*\|^2\Big]
		\end{eqnarray*}
		exists and that 
		\begin{eqnarray*}
			\underset{k\rightarrow \infty}\lim \Big[\|x_k-\bar{x}\|^2 -\theta\|x_{k-1}-\bar{x}\|^2-\beta \|x_{k-2}-\bar{x}\|^2 \Big]
		\end{eqnarray*}
		exists. These imply that
		$$
		\underset{k\rightarrow \infty}\lim \langle x_k-\theta x_{k-1}-\beta x_{k-2},x^*-\bar{x}\rangle
		$$
		\noindent exists. Now,
		\begin{eqnarray*}
			\langle \bar{x}-\theta \bar{x}-\beta \bar{x},x^*-\bar{x}\rangle&=& \underset{j\rightarrow \infty}\lim \langle x_{k_j}-\theta x_{k_j-1}-\beta x_{k_j-2},x^*-\bar{x}\rangle  \\
			&=& \underset{k\rightarrow \infty}\lim \langle x_k-\theta x_{k-1}-\beta x_{k-2},x^*-\bar{x}\rangle \\
			&=& \underset{m\rightarrow \infty}\lim \langle x_{k_m}-\theta x_{k_m-1}-\beta x_{k_m-2},x^*-\bar{x}\rangle  \\
			&=& \langle x^*-\theta x^*-\beta x^*,x^*-\bar{x}\rangle.
		\end{eqnarray*}
		Hence,
		$$
		(1-\theta-\beta)\|x^*-\bar{x}\|^2=0.
		$$
		\noindent Since $\beta \leq 0< 1-\theta$, we obtain that $x^*=\bar{x}$. Therefore, every accumulation point of $\{x_k\}_k$ is a solution to (EP). This completes the proof.

	\end{proof}

	\section{Numerical Experiment}\label{Numerical}
	
	In this section we present a numerical experiment to compare our proposed Algorithm to the one of  Grad, Lara and Marcavillaca  in \cite{GRAD3}. The codes were written  in  Python running on https://colab.research.google.com on a HP ProBook 430G6 Laptop with Windows 11 Home Single Language  and an Intel Core i5 8265U CPU with 1.60 GHz and 4.0GB RAM .
	
	
	The numerical example below shows the situation where the method proposed in this paper have a superior performance (in terms of the number of necessary iterations and CPU time until reaching the solution to (EP)) to the Relaxed-Inertial Proximal Point Algorithm proposed by Grad, Lara and Marcavillaca in \cite{GRAD3}.  As stopping criteria of the proposed algorithms we considered the situations when the norm, in particular the absolute value, of the
	difference between $y_k$ and $z_k$, not larger than an a priori given error $\varepsilon > 0$, i.e., we stop when $|z_k-y_k| < \varepsilon$.

	\begin{exm}\cite{GRAD3}
		Let $p, q\in \mathbb{N}, p>1, q>1$. Let $f:\R\times \R \longrightarrow \R $ be defined by 
		$$f(x, y) =p\big(\max\{\sqrt{|y|}, (y-q)^2-q\} \big) -p\big(\max\{\sqrt{|x|} (x-q)^2-q\}\big).$$
		
		As shown in \cite{LARA} the bifunction $f$ satisfies all assumptions $(Ai), i=1,2,3,4,5$ with values of parameters $\eta=1/2$, $\gamma=\dfrac{1}{2\sqrt{\delta^3}}$, where $\delta$ is a positive number greater than the positive solution of the equation $p(y-q)^2-\sqrt{|y|}=pq$. 
	\end{exm}
	We implements numerical experiments with values of parameters are chosen as follows: 
	$ p=2, q=99,$ $\epsilon= 1/4,  \rho =2/5, \varepsilon=10^{-13},\beta= - 0.0001, \theta=0.25,  \alpha=1/29-\varepsilon, \rho_k=1.4, \lambda_k=0.2 $ for all $k=1,2,3,\dots $ and with starting points $x^{-1}=3210, x_0=x_1=8297.$
	
	A straightforward verification shows that all assumptions 	 (Ci), $ i=1,2,3, 4$ hold according to the chosen value parameters. 
	

\begin{table}[ht]
	\centering
	\begin{tabular}{ccc}
		\hline 
		Algorithm & TIPPM & RIPPA-EP\\
		& $\varepsilon = 10^{-13}$& \\
		\hline 
		The solution & 79.1999999999766 & 79.199999999986\\
		\hline
		Time(s) & 0.029584646224975586 & 0.04530167579650879
		\\
		\hline
		Iterations & 152 & 313\\
		\hline 
	\end{tabular}
 \caption{Comparison of the performance of RTIPPA-EP(a), RIPPA-EP(b) with $\beta= - 0.0001, \theta=0.25,  \alpha=1/29-\varepsilon, \rho_k=1.4, \lambda_k=0.2 $ for all $k=1,2,3,\dots $.}
\end{table}
	
	\section{Conclusion}\label{conclude}
	\noindent
	In this paper, we showed that the Relaxation Proximal Point method with two-step inertial extrapolation can be adapted to solve non-convex equilibrium problems. We established the global convergence of the iterative sequence generated by the proposed method and provide a numerical illustration. Our numerical experiments highlight the advantages of using two-step inertial extrapolation over the one-step approach commonly employed in related studies on (convex) equilibrium problems. In future work, we plan to extend the proposed method by incorporating a combination of inertial and correction terms


\begin{thebibliography}{lll}
	\bibitem{Abubakar} Abubakar, J., Kumam, P., Hassan Ibrahim, A., Padcharoen, A: 	Relaxed Inertial Tseng’s Type Method for Solving the Inclusion Problem with Application to Image Restoration. Mathematics 8, 818 (2020).
	
	\bibitem{ANS}	 Ansari Q.H.,  Babu F.,  Raju M.S.: Proximal point method with Bregman distance for quasiconvex pseudomonotone equilibrium problems, Optimization, DOI: 10.1080/02331934.2023.2252430, (2023).	
		
	\bibitem{ANTI} Antipin, A.S.: Equilibrium programming: proximal methods. Comput. Mat. Math. Phys. 37, 1285–1296 (1997)
		
	\bibitem{ANTI2} Antipin, A.S.: The convergence of proximal methods to fixed points of extremal mappings and estimates of their rate of convergence. Comput. Math. Math. Phys. 35, 539–551 (1995)
		
	\bibitem{ANTI3} Antipin, A.S.: Non-gradient optimization of saddle functions. In: “Problems of Cybernetics. Methods and Algorithms for the Analysis of Large Systems”. Nauchn. Soviet po Probleme “Kibernetika”, Moscow (1987)
		
		
		
	

		
		
	
	\bibitem{BIGI2} Bigi, G., Castellani, M., Pappalardo, M., Passacantando, M.: Nonlinear Programming Techniques for
		Equilibria. Springer, Berlin (2019)
		
		\bibitem{BLU} Blum, E., Oettli, W.: From optimization and variational inequalities to equilibrium problems. Math. Stud. 63, 123–145 (1994)
		
									
		
		\bibitem{Bot3} Bo\c{t} R. I, Csetnek E. R.: An inertial forward-backward-forward primal-dual splitting algorithm for solving monotone inclusion problems. Numer. Algorithms 71, 519-540 (2016).
		
		
		
		\bibitem{BotRadu} Bo\c{t} R. I., Csetnek E. R.: An inertial Tseng's type proximal algorithm for nonsmooth and nonconvex optimization problems. J. Optim. Theory Appl. 171, no. 2, 600–616 (2016).
		
		\bibitem{BotSedlmayerVuong} Bo\c{t} R. I.: Sedlmayer, M. Vuong, P. T.
		A Relaxed Inertial Forward-Backward-Forward Algorithm for Solving Monotone Inclusions with Application to GANs. https://arxiv.org/abs/2003.07886
		
				
		
	

		\bibitem{CAM} {\rm  Cambini A.,  Martein L.:} Generalized Convexity and Optimization, {\it Springer,} Berlin (2009).
			
		\bibitem{CHEN} {\rm  Chen C.,  Ma S.,  Yang J.:} A general inertial proximal point algorithm for mixed variational inequality
		problem, {\it SIAM J. Optim.,} {\bf 25,} (2015) 2120-2142.
	
		\bibitem{CHE}  Cherukuri A.,  Gharesifard B.,  Cortes J.: Saddle-point dynamics: conditions for asymptotic stability of saddle points, SIAM J Control Optim, 55, 486--511, (2017).
	
		\bibitem{Cholamjiak} Cholamjiak W., Cholamjiak P., Suantai S.: An inertial forward-backward splitting method for solving inclusion problems in Hilbert spaces. J. Fixed Point Theory Appl. 20 , no. 1, Paper No. 42, 17  (2018).
	
	
	
		
		
		\bibitem{Combettes}   Combettes P.L. and  Glaudin L.E.: Quasi-nonexpansive iterations on the affine hull of orbits: from Mann’s mean value algorithm to inertial methods, SIAM J. Optim. 27, 2356–2380 (2017).
			
	

		\bibitem{FAN}	 Fan K.: A Minimax Inequality and Applications. In O. Shisha (ed.), Inequality III", pp. 103-113. Academic Press, New York, (1972).
			
		\bibitem{GRA} {\rm  Grad S.-M.,  Lara F.:} An extension of the proximal point algorithm beyond convexity, J. Glob. Optim. (2021). https://doi.org/10.1007/s10898-021-01081-4.
			
		\bibitem{Grad} {\rm  Grad S.-M.,  Lara F.:} Solving Mixed Variational Inequalities Beyond Convexity, {\it J. Optim. Theory Appl.,}  190,  565-580 (2021).
			
		\bibitem{GRAD3}  Grad S.-M.,  Lara F.,  Marcavillaca R. T.: Relaxed-inertial proximal point type algorithms for quasiconvex minimization, J. Global Optim., 85, 615--635, (2023).
				
	\bibitem{Grad}  Grad S. M, Lara F.,  Marcavillaca R.: Relaxed-Inertial Proximal point Algorithms for nonconvex equilibrium problems with applications, Journal of Optimization Theory and Applications, In press. hal-04429671 (2024).
				
		\bibitem{HAD} {\rm  Hadjisavvas N.,  Komlosi S.,  Schaible S.:} Handbook of Generalized Convexity and Generalized
				Monotonicity. Springer, Boston (2005).
				
				
	
			
		\bibitem{IUS} {\rm  Iusem A.,  Lara F.:} Optimality conditions for vector equilibrium problems with applications, {\it J. Optim. Theory Appl.,}  180, 187-206  (2019).
				
		\bibitem{IUSEM}  Iusem A.,  Lara F.: Proximal point algorithms for quasiconvex pseu-domonotone equilibrium problems, J. Optim. Theory Appl., 193, 443--461 (2022). 
				
				
		\bibitem{JOV}  Jovanov\'ic M.: A note on strongly convex and quasiconvex functions, Math. Notes, 60, 584--585, (1996).
					
					
	
					
				
		\bibitem{LARA}  Lara F.: On strongly quasiconvex functions: existence results and proxi-mal point algorithms, J. Optim. Theory Appl., 192, 891--911, (2022).
							
							
		\bibitem{Liang}  Liang J.: Convergence Rates of First-Order Operator Splitting Methods, Ph.D. thesis, Normandie University, France, GREYC CNRS UMR 6072, (2016).
							
							
							
							
							
\bibitem{nam3}  Mewomo O. T.,  Nwokoye R. N., Okeke  C. C.: Two-step inertial Tseng’s extragradient method for solving quasimonotone variational inequalities, Quaestiones Mathematicae, DOI: 10.2989/16073606.2024.2327562								
							
	\bibitem{MUU} {\rm  Muu L. D.,  Nguyen V. H.,  Quy N. V.:} On Nash–Cournot oligopolistic market equilibrium models with concave cost functions, {\it J. Glob. Optim.,}  41,  351-364 (2008).

\bibitem{MUU} Muu L. D., Oettli W.: Convergence of an adaptive penalty scheme for finding constraint equilibria. Nonlinear Anal. 18, 1159–1166 (1992).
							
							
							
							

	\bibitem{Padcharoen}  Padcharoen A.,  Kitkuan D.,  Kumam W., Kumam, P.: Tseng methods with inertial for solving inclusion problems and application to image deblurring and image recovery problems.  Comp and Math. Methods. DOI: 10.1002/cmm4.1088
	
								
	\bibitem{POL}  Polyak B. T.: Some methods of speeding up the convergence of iteration methods, U.S.S.R. Comput. Math. and Math. Phys. 4(5), 1–17 (1964).
							
	\bibitem{POL2}  Polyak B. T.: Introduction to optimization, New York, Optimization Software, Publications Division, (1987). 
							
	\bibitem{POL3}  Polyak B. T.: Existence theorems and convergence of minimizing sequences in extremum problems with restrictions, Sov Math Dokl, 7, 72{75}, (1966) (translation from Dokl Akad Nauk SSSR, 166, 287--290, (1966)).
								
	\bibitem{poon}   Poon C.,  Jingwei L.: Geometry of First-Order Methods and Adaptive Acceleration, (2003), arXiv:2003.03910		
	
						
		\bibitem{QUU} {\rm  Quoc T. D.,  Muu L. D.,  Hien N. V.,} Extragradient algorithms extended to equilibrium problems, {\it Optimization,}  57,  749-766 (2008).
								
								
		\bibitem{ROU}  Rouhani B. D.,  Piranfar M.R.: Asymptotic behavior for a quasi-autonomous gradient system of an expansive type governed by a  quasiconvex function, Electron. J. Di er. Equ., 2021, paper no. 15, (2021).
								
								
								
								
								
								
								
								
								
								
								
								
								
								
								

								
								
							
					
				
								
								
								
								
								
								
								
								
								
								
								
								
								
								

								
								
								
								
								
								
								
								
								
								

								
	\bibitem{Nam} Tran. V. N., Yekini S., Renqui X. Phan T. V.:  An inertial forward-backward-forward algorithm for non-convex mixed variational inequality, J. Appl. Numer. Optim. 5, No. 3, pp. 335-348 (2023).
	Available online at http://jano.biemdas.com
	https://doi.org/10.23952/jano.5.2023.3.04
								
	
								
		\bibitem{VIAL} Vial J.P.: Strong convexity of sets and functions, J. Math. Economics, 9, 187--205, (1982).
							
			
	\bibitem{MUU2}  Yen L. H.,  Muu L. D.: A parallel subgradient projection algorithm for quasiconvex equilibrium problems under the intersection of convex sets, Optimization, 71, 4447--4462, (2022).
										
	\bibitem{Yen1}  Yen L. H.,  Muu L. D.: An extragradient algorithm for quasiconvex		equilibrium problems without monotonicity, J. Global Optim., DOI: 								10.1007/s10898-023-01291-y, (2023).
								
		
								
\end{thebibliography}
		\end{document}